\documentclass[10pt, a4paper]{amsart}

\usepackage[utf8]{inputenc}
\usepackage[T1]{fontenc}

\usepackage{eurosym}
\usepackage{amsmath}
\usepackage{mathrsfs}
\usepackage{amsfonts}
\usepackage{graphicx}
\usepackage{color}
\usepackage{amsfonts}
\usepackage{amssymb}%
\usepackage{esint}
\usepackage{enumerate}

\usepackage[abbrev]{amsrefs}

\newtheorem{theorem}{Theorem}[section]

\newtheorem{corollary}{Corollary}[section]

\newtheorem{lemma}{Lemma}[section]

\newtheorem{proposition}[theorem]{Proposition}
\newtheorem{remark}{Remark}[section]

\newtheorem{openquestion}{Open Question}
\numberwithin{equation}{section}

\def\R{{\mathbb R}}
\def\N{{\mathbb N}}

\usepackage{csquotes}
\newcommand{\rlemma}[1]{Lemma~\ref{#1}}
\newcommand{\rth}[1]{Theorem~\ref{#1}}

\newcommand{\rprop}[1]{Proposition~\ref{#1}}

\def\st{{\mathbb S}^2}

\def\SNO{\omega_{N-1}}

\def\H1O{H^1(\Omega;\st)}
\def\dist{\operatorname{dist}}
\def\Dist{\operatorname{Dist}}
\def\div{\operatorname{div}}

\def\DistH1{\Dist_{H^1}}
\def\distH1{\dist_{H^1}}

 \newcommand{\Prod}{\mathop{\prod}\limits}

\mathchardef\mhyphen="2D

\title[Best constants for critical exponent Sobolev embeddings]{Best constants for two families of higher order critical Sobolev embeddings}

\author[I. Shafrir]{Itai Shafrir}
\address{Department of Mathematics\hfill\break\indent 
Technion - Israel Institute of Technology\hfill\break\indent
32000 Haifa, Israel\hfill\break\indent }
\email{shafrir@math.technion.ac.il}

\author[D. Spector]{Daniel Spector}
\address{Department of Applied Mathematics\hfill\break\indent 
National Chiao Tung University\hfill\break\indent
1001 Ta Hsueh Road.\hfill\break\indent 
Hsinchu, Taiwan R.O.C.}
\address{
National Center for Theoretical Sciences\hfill\break\indent 
National Taiwan University\hfill\break\indent
No. 1 Sec. 4 Roosevelt Rd.\hfill\break\indent 
Taipei, 106, Taiwan R.O.C.}
\email{dspector@math.nctu.edu.tw}

\begin{document}

\subjclass[2010]{Primary 46E35, Secondary 35A23}

\maketitle

\begin{abstract}
In this paper we obtain the best constants in some higher order Sobolev inequalities in the critical exponent.  These inequalities can be separated into two types:  those that embed into $L^\infty(\mathbb{R}^N)$ and those that embed into slightly larger target spaces.  Concerning the former, we show that for $k \in \{1,\ldots, N-1\}$, $N-k$ even, one has an optimal constant $c_k>0$ such that
\[ \|u\|_{L^\infty} \leq c_k \int |\nabla^k (-\Delta)^{(N-k)/2} u|\]
for all $u \in C^\infty_c(\mathbb{R}^N)$ (the case $k=N$ was handled in \cite{Shafrir:2017}). Meanwhile the most significant of the latter is a variation of D. Adams' higher order inequality of J. Moser:  For $\Omega \subset \mathbb{R}^N$, $m \in \mathbb{N}$ and $p=\frac{N}{m}$, there exists $A>0$ and optimal constant $\beta_0>0$ such that
\[ \int_{\Omega} \exp (\beta_0 |u|^{p^\prime}) \leq A |\Omega| \]
for all $u$ such that $\|\nabla^m u\|_{L^p(\Omega)} \leq 1$, where $\|\nabla^m u\|_{L^p(\Omega)}$ is the traditional semi-norm on the space $W^{m,p}(\Omega)$.
 \end{abstract}

\section{Introduction and Main Results}
Let $\Omega \subset \mathbb{R}^N$ be open, bounded and smooth or all of $\mathbb{R}^N$, $m \in \mathbb{N}$, and $p \in [1,\infty)$.  Further denote by $W^{m,p}(\Omega)$ the Sobolev space of functions in $L^p(\Omega)$ whose $m$th order distributional derivatives are also in $L^p(\Omega)$.  It is well known that the case $mp=N$ is a limiting one for the embedding of $W^{m,p}(\Omega)$ into $L^\infty(\Omega)$. Indeed, the embedding holds for $mp>N$ and fails to hold for $mp<N$. The critical exponent $mp=N$ is more delicate. The exception here is $W^{N,1}(\Omega)$, for which the embedding in $L^\infty(\Omega)$ does hold.  For all other values of $p=N/m>1$, $W^{m,p}(\Omega)$ does not embed into $L^\infty(\Omega)$.  In this regime, instead of boundedness one can show the local integrability of $\exp({a|u|^{p^\prime}})$ for some $a>0$. Our main concern here is the issue of {\em best constants} in the aforementioned embeddings and related ones.

We start with the exceptional case $p=1$. In a recent paper \cite{Shafrir:2017}, the first author has found the best constant for the embedding of the space ${W}^{N,1}(\mathbb{R}^N)$ into $L^\infty(\mathbb{R}^N)$. That is, he computed the 
 smallest constant $c_N>0$ for which the inequality
\begin{align}
\sup_{x \in \mathbb{R}^N} |u(x)| \leq c_N \int_{\mathbb{R}^N} |\nabla^N u(x)|\,dx\, \label{IS}
\end{align}
holds for all $u \in {W}^{N,1}(\mathbb{R}^N)$.  An earlier result by Humbert and Nazaret~\cite{hn} established this results in dimension $N=1,2$.  Moreover, it was proved in \cite{hn} for $N=2$ and in \cite{Shafrir:2017}  for $N\geq2$, that  the inequality in \eqref{IS} is strict for any non-trivial $u$ (i.e., if $u$ is not the zero function). Here and in the sequel, for $k \in \mathbb{N}$, we denote by 
\begin{equation*}
\nabla^k u(x)=\left\{\frac{\partial^k u}
{\partial_{x_{i_1}}\dots\partial_{x_{i_k}}}\right\}_{i_1,\ldots,i_k\in {\mathcal I}_N } (\text{where }{\mathcal I}_N=\{1,\ldots,N\})\,,
\end{equation*}
the tensor consisting of the $N^k$  partial derivatives of $u$ of order $k$ at the point $x$ and by 
 $|\nabla^k u|(x)$ the Euclidean norm of this vector in  $\R^{N^k}$. Note that by scaling invariance proporties of the quantities involved, \eqref{IS} implies, and actually is equivalent to, an analogous inequality for functions in $W^{N,1}_0(\Omega)$ for {\em any} subdomain $\Omega\subset\R^N$. 
 
 The key observation in \cite{Shafrir:2017} that led to \eqref{IS} was the identification of $\log |x|$ as a fundamental solution of  a certain elliptic operator of order $2N$, namely  
 \begin{align}
 \label{eq:shaf}
 (-1)^{N} \div_N \left( |x|^N \nabla^N \log |x| \right) = \mu_N \delta_0.
 \end{align}
 Above and in the sequel we denote, for a tensor $T$, consisting of $N^k$ components, 
 \begin{equation}\label{eq:div_k}
 \div_k(T)=\sum_{i_1,\ldots,i_k\in {\mathcal I}_N}\frac{\partial^k T_{i_1,\ldots,i_k}}
 {\partial_{x_{i_1}}\dots\partial_{x_{i_k}}}\,.
 \end{equation}
 One of our main observations here is that \eqref{eq:shaf} is just a special case of  a family of equations satisfied by $\log |x|$, namely
 \begin{align}
 (-1)^k \div_k (-\Delta)^{(\alpha-k)/2} \left(|x|^{2\alpha-N} \nabla^k (-\Delta)^{(\alpha-k)/2} \log |x|\right) = \mu_{k,\alpha} \delta_0 \label{log_PDE}
 \end{align}
for $k\in\N$ and $\alpha\in[k,\infty)$. Actually, in order to let \eqref{log_PDE}  make sense also when $(\alpha-k)/2$ is not an integer, we consider an equivalent form of it in \eqref{eq:37} below.
From \eqref{log_PDE} we deduce below in Theorem~\ref{thm2} that for each $k \in \{1,\ldots,N-1\}$ one has the following  inequality:
 \begin{align}
 \sup_{x \in \mathbb{R}^N} |u(x)| \leq c_k\int_{\mathbb{R}^N} |\nabla^k (-\Delta)^{(N-k)/2} u(x)|\,dx\,, \label{eq:k_inequality}
 \end{align}
 for all $u \in C^\infty_c(\mathbb{R}^N)$.  In the case where $N-k$ is even we prove that the constant $c_k$ 
 we have found in \eqref{eq:k_inequality} is also optimal. We note that in the case $k=1$ inequality \eqref{eq:k_inequality} could alternatively be deduced from a recent work of the second author and Rahul Garg concerning the mapping properties of the Riesz potential \cites{Garg-Spector:2015,Garg-Spector:2015-2} (although this inequality was not stated explicitly in \cites{Garg-Spector:2015,Garg-Spector:2015-2}, it follows from Lemma 3.1 in \cite{Garg-Spector:2015-2} and the relationship of the fractional Laplacian and Riesz potentials - see below for precise definitions).

Our results extend in the usual way to the completion of $C^\infty_c(\R^N)$ with respect to the semi-norm
\begin{align*}
| u|_{X^{N,k}} := \int_{\mathbb{R}^N} |\nabla^k (-\Delta)^{(N-k)/2} u(x)|\,dx\,.
\end{align*}
If we denote these spaces by $\{X^{N,k}\}_{k=1}^N$, then one observes they can be partitioned into two nested families of spaces
\begin{equation}
\label{eq:nested}
\begin{aligned}
\dot{W}^{N,1}(\mathbb{R}^N) &= X^{N,N} \subset X^{N,N-2} \subset\dots \subset X^{N,N \text{ mod } 2} \\
&\;\;\;\;X^{N,N-1} \subset X^{N,N-3} \subset\dots \subset X^{N,N-1\text{ mod }2}.
\end{aligned}
\end{equation}
Simple pointwise inequalities show that one has the preceding inclusions, while the fact that all the inclusions in \eqref{eq:nested} are actually {\em strict} is a direct consequence of Ornstein's celebrated theorem, see \cite{ornstein} (we are indebted to Petru Mironescu for informing us about this result).  

Thus if one rewrites \eqref{eq:k_inequality} as the family of inequalities
\begin{align}
\sup_{x \in \mathbb{R}^N} |u(x)| \leq c_k| u|_{X^{N,k}},\label{k_inequality}
\end{align}
for $k \in \{1,\ldots,N\}$, our work asserts that such a $c_k<+\infty$ (which could also be deduced from classical potential representations and Sobolev embeddings on the Lorentz scale), and in the case $N-k$ is even, it cannot be improved.  In this framework it is natural also to consider the case $k=0$, i.e.~the space $X^{N,0}$  associated with the semi-norm 
\begin{align*}
| u|_{X^{N,0}} := \int_{\mathbb{R}^N} |(-\Delta)^{N/2} u(x)|\,dx.
\end{align*}
 Here,  for $k=0$, in strict contrast with the case $k\ge1$, the analogue to \eqref{k_inequality} is {\em false}, as $X^{N,0}$ does not embed into $L^\infty(\R^N)$. We will discuss in more detail this issue below.

In order to give the explicit expression for $c_k$ in \eqref{eq:k_inequality} we need to recall the work of Morii, Sato, and Sawano \cite{MoriiSatoSawano} on the Euclidean norm of the derivatives of certain radial functions.  In particular we require firstly their result concerning the function $\log |x|$ that
\begin{align}
| \nabla^k \log |x||^2 = \frac{\ell^k_N}{|x|^{2k}},\;\; x \neq 0,\label{log_equality}
\end{align}
for a combinatorial constant $\ell_N^k$ (see \eqref{eq:6} for  its explicit value), and additionally the following:  for each integer $k\ge 1$ and $s\in\R$ 
there is a positive constant  $\lambda^{s,k}_N$ (denoted by $\gamma^{s,k}_N$ in
\cite{MoriiSatoSawano}) that satisfies
\begin{equation}
\label{eq:5}
| \nabla^k |x|^s|^2 = \frac{\lambda^{s,k}_N}{|x|^{2(k-s)}},\;\; x \neq 0\,,
\end{equation}
 see \eqref{eq:7}.
With these ingredients we can give the following result on the sharp constant in the inequalities \eqref{k_inequality}.
\begin{theorem}\label{thm2}
For  $k \in \{1,\ldots, N-1\}$  set $c_k:={\left(\lambda_N^{k-N,k}\right)}^{-1/2}{\gamma(N-k)}^{-1}$.  Then,
\begin{equation}
\sup_{x \in \mathbb{R}^N} |u(x)| \leq c_k\int_{\mathbb{R}^N} |\nabla^k (-\Delta)^{(N-k)/2} u(x)|\;dx \label{k_inequality_star}
\end{equation}
for all $u \in C^\infty_c(\mathbb{R}^N)$.  Furthermore, when $N-k$ is even $c_k$ is optimal in the sense that it cannot be replaced by any smaller constant. 
\end{theorem}
\noindent
Here, $\gamma(\alpha)$ is a normalization constant associated with the Riesz potential $I_\alpha$ ($\alpha\in(0,N)$), for which we use the definition as given in  \cite{Stein:1970}*{p. 117}):
\begin{equation}
\label{eq:I-alpha}
I_\alpha g(x)  \equiv (I_\alpha\ast g)(x) := \frac{1}{\gamma (\alpha)} 
\int_{\mathbb{R}^N} \frac{g(y)}{|x-y|^{N-\alpha}}\;dy,
\end{equation}
where
\begin{equation}
\label{eq:gamma}
\gamma(\alpha) = 
\pi^{N/2}2^{\alpha} 
\Gamma(\tfrac{\alpha}{2})/\Gamma(\tfrac{N-\alpha}{2}).
\end{equation}
For later use we also define, this time for each $\alpha\in[0,\infty)$,
 \begin{equation*}
\widetilde{\gamma}(\alpha)=\begin{cases}
\alpha\gamma(\alpha) & \alpha>0\\
\SNO & \alpha=0
\end{cases}
.
\end{equation*}
Note that $\widetilde{\gamma}(\alpha)$ is continuous at $\alpha=0$. We were only able to prove optimality of the constant $c_k$ in \eqref{k_inequality_star} when $N-k$ is even. It is therefore natural to raise the following question:
\begin{openquestion}
Is the constant in \eqref{k_inequality_star} optimal when $N-k$ is odd?
\end{openquestion}

Next we move to the question of optimal constants in embedding of the space $W^{m,p}_0(\Omega)$ when $p=N/m>1$. The first result of this type is due to Moser~\cite{moser}, who for the case $m=1,p=N$ proved the following: 
   \begin{equation}\label{eq:moser}
   \int_{\Omega} \exp(\alpha_0(N)|u|^{p^\prime})\leq c|\Omega|,~\forall u\in W^{1,N}_0(\Omega)\text{ s.t. }\int_{\Omega}|\nabla u|^N\leq 1,
   \end{equation}
   with $\alpha_0(N)=N\SNO^{\frac{1}{N-1}}$, and $\alpha_0(N)$ cannot be replaced by any larger number.
   
 In \cite{Adams:1988} D.~Adams generalized Moser's result to all spaces $W_0^{m,p}(\Omega)=W_0^{m,N/m}(\Omega)$, where $m$ is any positive integer less than $N$, for which he proved that
   \begin{equation}\label{eq:adams}
   \int_{\Omega} \exp(\beta_0(N,m)|u|^{p^\prime})\leq c|\Omega|,~\forall u\in W^{m,p}_0(\Omega)\text{ s.t. }\int_{\Omega}|D^m u|^p\leq 1,
   \end{equation}
   with $\beta_0(N,m)$ given by 
   \begin{equation}
   \label{eq:beta}
   \beta_0(m,N)=\begin{cases} \frac{N}{\omega_{N-1}}\gamma(m)^{p^\prime} & m \text{ even }\\
  \frac{N}{\omega_{N-1}}{\widetilde{\gamma}(m-1)}^{p^\prime} & m \text{ odd }
   \end{cases}
   .
   \end{equation}
 Moreover,  $\beta_0(N,m)$ is optimal, in the sense that it cannot be replaced by any larger number. In the above,
  \begin{equation}\label{Adams_derivative}
  D^m u =\begin{cases}
  (-\Delta)^{m/2} u& m \text{ even}\\
  \nabla (-\Delta)^{(m-1)/2} u& m \text{ odd}
  \end{cases} 
  .
  \end{equation}
  The norm $\|D^mu\|_{L^p(\Omega)}$ used by Adams looks somewhat unnatural, as in particular it requires to distinguish between the cases $m$ is even and odd in \eqref{eq:beta}.
  We shall see below how to obtain the same result as Adams', but for the more traditional norm $\||\nabla^m u|\|_{L^p(\Omega)}$.
  
The difficult part of Adams' proof of \eqref{eq:adams} is the following sharp exponential estimate for the Riesz potential proven in \cite{Adams:1988}:
 
\begin{theorem}\label{Adams}[D. Adams]
For $1<p<+\infty$, there is a constant $A=A(p)$ such that for all $f \in L^p(\mathbb{R}^N)$ with support contained in $\Omega$, $|\Omega|<+\infty$, 
\begin{equation}
\label{eq:adam-potential-ineq}
\int_\Omega \exp\left (\frac{N}{\omega_{N-1}}\gamma(\alpha)^{p^\prime} \left| \frac{I_\alpha f(x)}{\|f\|_{L^p}}\right|^{p^\prime}\right)\;dx \leq A |\Omega|,
\end{equation}
where $\alpha = N/p$.  Furthermore, no number greater than $\frac{N}{\omega_{N-1}}\gamma(\alpha)^{p^\prime}$ can replace the coefficient without forcing $A$ to depend on $u$ as well as $p$.
\end{theorem}
In order to deduce \eqref{eq:adams} from \rth{Adams} Adams used  standard potential representations of a function in terms of the differential object \eqref{Adams_derivative}.  Indeed, one has for any smooth function with compact support:
\begin{equation}
\label{eq:rep}
u(x)=\begin{cases}\frac{1}{\gamma(m)}\int_{\R^N} (-\Delta)^{m/2}u(y)|x-y|^{m-N}\,dy & m \text{ even}\\
\frac{1}{\widetilde{\gamma}(m-1)}\int_{\R^N} \nabla\left((-\Delta)^{(m-1)/2}u(y)\right)\cdot|x-y|^{m-1-N}(x-y)\,dy & m \text{ odd}
\end{cases}
.
\end{equation}
 Note that the special case $m=1$ in \eqref{eq:rep} is nothing but the well-known formula:
 \[ 
 u(x)=\frac{1}{\SNO}\int_{\R^N} \frac{\nabla u(y)\cdot(x-y)}{|x-y|^{N}}\,dy\,. 
  \] 
In order to justify \eqref{eq:rep} for every $m\ge2$, note that the first formula  just expresses the fact that $I_m$ and $(-\Delta)^{m/2}$ are the inverse of each other. To obtain the second formula (for $m\ge3$ odd) it suffices to apply the first formula for $m-1$ instead of $m$, use the identity 
\[ 
\div\left(|x-y|^{m-1-N}(x-y)\right)=(m-1)|x-y|^{m-1-N}\,,
\] 
and finally apply integration by parts.
From \eqref{eq:rep} one deduce easily the pointwise inequality 
\begin{equation}\label{eq:pointw}
|u(x)|\leq \begin{cases}I_m(|D^m u|)(x) & m \text{ even}\\
 \frac{\gamma(m)}{\widetilde{\gamma}(m-1)}I_m(|D^m u|)(x) & m \text{ odd}
\end{cases}
.
\end{equation}
Plugging \eqref{eq:pointw} in \eqref{eq:adam-potential-ineq} leads immediately to \eqref{eq:adams}. 
In our approach, we still rely on \rth{Adams}, but  instead of \eqref{eq:pointw} we use the inequality
\[
 |u(x)| \leq \frac{\gamma(m)}{\sqrt{\ell^m_N}\SNO} I_m(|\nabla^m u|)(x)\,,
 \] 
 which is established in Corollary~\ref{endpoint_derivative}. This inequality is a consequence of the new representations of the higher order gradient in terms of potentials that are sharp, which follow from the case $k=m=\alpha$ of our equations \eqref{log_PDE}. 
 \par The above considerations allow us to obtain the following variant  of Adams' estimate \eqref{eq:adams}:
\begin{theorem}\label{thm4}
Let $m \in \{1,\ldots,N-1\}$ and $p=N/m \in (1,N]$, and define 
\begin{align*}
\widetilde\beta_0(m,N):= N \omega_{N-1}^{m/(N-m)} (\ell^m_N)^\frac{N}{2(N-m)}.
\end{align*}
Then we have
\begin{align*}
\int_{\Omega} e^{\widetilde\beta_0(m,N) |u|^{p^\prime}}\;dx \leq A |\Omega|,
\end{align*}
for all  $u\in W^{m,p}_0(\Omega)$ such that $\||\nabla^m u| \|_{L^p(\Omega)}\leq 1$. Furthermore, no number greater than $\widetilde\beta_0(m,N)$ can replace the coefficient without forcing $A$ to depend on $u$ as well as $p$.
\end{theorem}
We draw the attention of the reader to the natural way in which our formula for $\widetilde\beta_0(m,N)$ generalizes Moser's formula for $\alpha_0(N)$ in \eqref{eq:moser}.
\begin{remark}
The only case where Adams' constant coincides with ours is when $N=2m$.  This is because of the equality
\begin{equation}
\label{eq:eqnorms}
\|D^m u\|_{L^2(\Omega)} = \||\nabla^m u|\|_{L^2(\Omega)}.
\end{equation}
The equality $\beta_0(m,2m)=\widetilde{\beta}_0(m,2m)$ leads, after some simple manipulations, to the formula
\begin{equation}
\label{eq:ell-m-2m}
\ell_{2m}^m = 2^{2(m-1)} ((m-1)!)^2.
\end{equation}
An elementary direct way to deduce \eqref{eq:ell-m-2m} from \eqref{eq:eqnorms} is to apply the latter to the family $\{u_\varepsilon\}$ constructed in \rprop{prop:prop} (functions that approximate $\log(1/|x|)$). A simple computation gives 
\begin{align*}
\int_{\R^N}  |\nabla^m u_\varepsilon|^2=\SNO\ell_{2m}^m\log(1/\varepsilon)+O(1),\\
\intertext{ while }
\int_{\R^N}  |D^m u_\varepsilon|^2=\SNO 2^{2(m-1)} ((m-1)!)^2\log(1/\varepsilon)+O(1).
\end{align*}
Equating the above expressions and sending $\varepsilon$ to zero yields \eqref{eq:ell-m-2m}.
\end{remark}

More generally, the potential representations we develop in this paper enables us to give the best constants in the intermediate family, which is given in our
\begin{theorem}\label{thm5}
Let $m \in \mathbb{N}$, $k \in \{1,\ldots, m-2\}$, $m-k$ even, and suppose $p=N/m \in (1,\infty)$.  Further define 
\begin{align*}
\widetilde\beta_0(m,k,N):=\frac{N}{\omega_{N-1}} (\gamma(m-k)\sqrt{\lambda_N^{k-m,k}})^{p^\prime}.
\end{align*}
Then we have
\begin{align*}
\int_{\Omega} e^{\widetilde\beta_0(m,k,N) |u|^{p^\prime}}\;dx \leq A |\Omega|,
\end{align*}
for all $u\in C^\infty_c(\Omega)$ such that $\||\nabla^k (-\Delta)^{(m-k)/2} u| \|_{L^p}\leq 1$.  Furthermore, no number greater than $\widetilde\beta_0(m,k,N)$ can replace the coefficient without forcing $A$ to depend on $u$ as well as $p$.
\end{theorem}

Finally, we return to the exceptional case $X^{N,0}(\R^N)$. Here the embedding into $L^\infty$ is false, though it is well known that one has an embedding into the larger space of functions of bounded mean oscillation (BMO).  In fact, in this regime John and Nirenberg's work implies the existence of a $C>0$ for which
\begin{align}
|u|_{BMO} \leq C \int_{\mathbb{R}^N} |(-\Delta)^{N/2} u(x)|\;dx \label{k_equals_0}
\end{align}
for all $(-\Delta)^{N/2}u \in L^1(\mathbb{R}^N)$.  Indeed, they show that functions which can be expressed as
\begin{align}
u(x) =  \frac{2}{\pi^{N/2}2^{N}  \Gamma(\tfrac{N}{2})} \int_{\mathbb{R}^N} \log |x-y|^{-1} f(y) \;dy, \label{JN}
\end{align}
for some $f \in L^1(\mathbb{R}^N)$ are of bounded mean oscillation (see p.~ 417 in \cite{JohnNirenberg}) with a norm depending on the norm of $f \in L^1(\mathbb{R}^N)$.  The inequality \eqref{k_equals_0} follows when one takes into account that for such $u$ one has
\begin{align*}
(-\Delta)^{N/2} u = f,
\end{align*}
(see, for example, Corollary~\ref{corr:delta} below).

We can give an alternative proof of this embedding with the techniques developed here, provided one uses the appropriate semi-norm on $BMO$.  In particular, let us here take the natural norm on $BMO$ arising as the dual of a Banach space.  Thus, we take for granted C. Fefferman's result that this space is the dual of the Hardy space $\mathcal{H}^1(\mathbb{R}^N)$ (see \cites{Fefferman, FeffermanStein})
\begin{align*}
\left(\mathcal{H}^1(\mathbb{R}^N)\right)^\prime = BMO(\mathbb{R}^N),
\end{align*}
where we equip the Hardy space $\mathcal{H}^1(\mathbb{R}^N)$ with the norm of Stein and Weiss \cite{SteinWeiss}
\begin{align*}
\|f\|_{\mathcal{H}^1(\mathbb{R}^N)} := \int_{\mathbb{R}^N} \left|\left(f(x),Rf(x)\right)\right| \;dx,
\end{align*}
where $Rf=\nabla(I_1f)$ is the Riesz transform of $f$.  Then we consider the semi-norm of an element of $u \in BMO(\mathbb{R}^N)$ as
\begin{align*}
|u|_{BMO} &= \sup_{f \in \mathcal{H}^1(\mathbb{R}^N), \|f\| \leq 1} \int_{\mathbb{R}^N} uf. 
\end{align*}

For this semi-norm, we prove
\begin{proposition}\label{thm3}
Define $c_0:=c_1$ as in Theorem \ref{thm2} in the case $k=1$.  Then one has
\begin{align}
|u|_{BMO} \leq c_0 \int_{\mathbb{R}^N} |(-\Delta)^{N/2} u(x)|\;dx \label{k_0}\,,
\end{align}
for all $u$ such that $(-\Delta)^{N/2}u \in M_b(\mathbb{R}^N)$, where $M_b(\mathbb{R}^N)$ the space of finite Radon measures.
\end{proposition}

In this endpoint, we are not able to prove optimality, and so this prompts one to ask
\begin{openquestion}
Is the constant in \eqref{k_0} optimal?
\end{openquestion}

The plan of the paper is as follows.  In Section \ref{preliminaries} we collect some useful results regarding the function $\log|x|$, the Riesz potentials, and the fractional Laplacian.  In Section \ref{PDEs} we state and prove precise versions of the equation \eqref{log_PDE}.  In Section \ref{proofs} we prove our main embedding results, with best constants, whenever possible.

\section{Preliminaries}\label{preliminaries}

We here recall some facts which will be useful in the sequel.  For $\alpha \in (0,N)$,  the Riesz potential as defined in \eqref{eq:I-alpha}
satisfies the semigroup property 
\begin{align*}
  I_{\alpha+\beta}g&=I_\alpha I_\beta g, &\text { 
for } \alpha,\beta>0, \text{ such that  } \alpha+\beta<N, 
\end{align*}
for $g$ in a suitable class of functions.   Note that this implies (or can be deduced from) the Fourier space relation
\begin{align*}
(I_\alpha g)\widehat{\phantom{x}} (\xi) = (2\pi |\xi|)^{-\alpha} \widehat{g}(\xi),
\end{align*}
again for $g$ suitably regular and integrable, and where we take the convention that
\begin{align*}
 \widehat{g}(\xi) = \int_{\mathbb{R}^N} g(x) e^{-2\pi i x\cdot \xi}\;dx.
\end{align*}

 For such functions, we can define the inverse of $I_\alpha$ by
\begin{align*}
((-\Delta)^{\alpha/2}g)\widehat{\phantom{x}}(\xi) := (2\pi |\xi|)^\alpha \widehat{g}(\xi).
\end{align*}
This is the fractional Laplacian, and in particular, when $0<\alpha<2$, one can deduce from these definitions  the relation
\begin{align*}
(-\Delta)^{\alpha/2}g = I_{2-\alpha} (-\Delta) g.
\end{align*}

A limiting case of the Riesz potentials $I_\alpha$ is the case $\alpha \to N^-$, the limit interpreted in a suitable sense.  In fact, for $f \in \mathcal{S}(\mathbb{R}^N)$ with $\int_{\R^N}  f=0$ one has
\begin{align*}
\lim_{\alpha \to N^-} I_\alpha f &= \lim_{\alpha \to N^-} \int_{\mathbb{R}^N} \frac{1}{\gamma(\alpha)} \frac{1}{|x-y|^{N-\alpha}} f(y)\;dy\\
&= \lim_{\alpha \to N^-} \int_{\mathbb{R}^N} \frac{1}{\gamma(\alpha)} \left[\frac{1}{|x-y|^{N-\alpha}}- \frac{1}{|x|^{N-\alpha}}\right]f(y)\;dy,
\end{align*}
while
\begin{align*}
\frac{1}{|x-y|^{N-\alpha}}- \frac{1}{|x|^{N-\alpha}} &= e^{\ln |x-y|^{\alpha-N}}- e^{\ln |x|^{\alpha-N}}\\
&= (\alpha-N)\left( \ln |x-y| - \ln |x|\right)+ o(\alpha-N).
\end{align*}
Thus applying Lebesgue's dominated convergence theorem one obtains
\begin{align*}
\lim_{\alpha \to N^-} I_\alpha f = \lim_{\alpha \to N^-} \frac{(\alpha-N)}{ \gamma(\alpha)} \int_{\mathbb{R}^N} \ln |x-y| f(y)\;dy.
\end{align*}
Finally, we can resolve the constant
\begin{align*}
\lim_{\alpha \to N^-} \frac{(\alpha-N)}{ \gamma(\alpha)} &= \lim_{\alpha \to N^-} (\alpha-N) \frac{\Gamma(\tfrac{N-\alpha}{2})}{\pi^{N/2}2^{\alpha} 
\Gamma(\tfrac{\alpha}{2})} \\
&= \frac{-2}{\pi^{N/2}2^{N}  \Gamma(\tfrac{N}{2})}  \lim_{\alpha \to N^-} \tfrac{(N-\alpha)}{2} \Gamma(\tfrac{N-\alpha}{2}),
\end{align*}
and thus we find
\begin{align*}
\lim_{\alpha \to N^-} I_\alpha f =   \frac{2}{\pi^{N/2}2^{N}  \Gamma(\tfrac{N}{2})} \int_{\mathbb{R}^N} \ln |x-y|^{-1} f(y)\;dy,
\end{align*}
so that in a suitable sense
\begin{align*}
(I_Ng)\widehat{\phantom{x}} (\xi) = (2\pi |\xi|)^{-N} \widehat{g}(\xi).
\end{align*}
\begin{corollary}
	\label{corr:delta}
	One has, in the sense of distributions,
	\begin{align*}
	(-\Delta)^{N/2}  \frac{2}{\pi^{N/2}2^{N}  \Gamma(\tfrac{N}{2})}  \log |x|^{-1}=\delta_0.
	\end{align*}
\end{corollary}
For $m\in(0,N)$ we compute $(-\Delta)^{m/2}\log |x|$ in the following
\begin{lemma}
  \label{lem:lap-log}
 For any $m \in (0,N)$ we have in the sense of distributions
 \begin{equation}
   \label{eq:1}
   (-\Delta)^{m/2}\log |x|=-\left(\frac{\gamma(m)}{\omega_{N-1}}\right)\frac{1}{|x|^m}\,.
 \end{equation}
\end{lemma}

\begin{proof}
We treat separately the cases $N\ge 3$ and $N=2$. For $N\ge3$ we recall that
  \begin{equation}
    \label{eq:4}
    -\Delta\log |x|=(2-N)/|x|^2 = (2-N) \gamma(N-2) (I_{N-2}\delta_0)(x) \,.
  \end{equation}
Using the representation of
  $(-\Delta)^{-1/2}$ via the Riesz potential we get the following
  formula for $(-\Delta)^{1/2}\log |x|$:
  \begin{multline}
    \label{eq:3}
    (-\Delta)^{1/2}\log
    |x|= I_1 (-\Delta) \log |x| = I_1 \left((2-N)\gamma(N-2) I_{N-2}\delta_0\right) (x)\\= (2-N) \gamma(N-2) (I_{N-1}\delta_0)(x).
  \end{multline}
 The semi-group property of the Riesz potential and fractional Laplacian now imply that 
 \begin{align}
 \label{eq:999}
    (-\Delta)^{m/2}\log |x| = (2-N)\gamma(N-2) (I_{N-m}\delta_0)(x)= \frac{(2-N) \gamma(N-2)}{\gamma(N-m)} \frac{1}{|x|^m}\,.
 \end{align}
Next we use the well-known formula
 \begin{equation*}
 \omega_{N-1}=\frac{2\pi^{N/2}}{\Gamma(N/2)}=\frac{\gamma(2)}{N-2}~(\text{see
 	\eqref{eq:gamma}}),
 \end{equation*}
 in conjunction with the semigroup property of $\gamma$ to
 rewrite the coefficient on the R.H.S.~of \eqref{eq:999} as
 \begin{equation}
 \label{eq:35}
 \frac{(2-N) \gamma(N-2)}{\gamma(N-m)}=-\frac{\gamma(2)\gamma(N-2)}{\omega_{N-1}\gamma(N-m)}
 =-\frac{\gamma(m)}{\omega_{N-1}}\,,
 \end{equation}
 and \eqref{eq:1} follows, when $N\ge3$.
 \par When $N=2$ we start with 
 \begin{equation*}
 -\Delta\log |x|=-2\pi\delta_0\,,
 \end{equation*}
  then, again by the semi-group property of the Riesz potential and of $\gamma$ we deduce that
  \begin{equation*}
  (-\Delta)^{m/2}\log|x|=I_{2-m}(-\Delta\log|x|)=-2\pi I_{2-m}\delta_0=-\frac{2\pi}{\gamma(2-m)|x|^m}.
  \end{equation*}
  Finally, it suffices to note that 
  $$
  \gamma(2-m)=\frac{\gamma(1)^2}{\gamma(m)}=\frac{4\pi^2}{\gamma(m)}\,,
  $$
  to conclude that $\frac{2\pi}{\gamma(2-m)}=\frac{\gamma(m)}{2\pi}=\frac{\gamma(m)}{\omega_1}$.
 \end{proof}
 
In the paper of Morii, Sato, and Sawano \cite{MoriiSatoSawano}, they verify that the identities \eqref{log_equality} and \eqref{eq:5} hold and compute explicitly the constants in terms of some combinatorial quantities.  We here recall their values.  If we denote by
\begin{equation*}
  (\nu)_k=
  \begin{cases}
    \Prod_{j=0}^{k-l}(\nu-j)&\text{ for }\nu\in\R, k\in\N\\
   1 &\text{ for }\nu\in\R, k=0
  \end{cases}
  ,
\end{equation*}
then the relevant quantities are firstly
\begin{align}
    \label{eq:6}
    \ell_N^m&=m!\sum_{l=0}^{\left \lfloor{m/2}\right \rfloor}
    (m-2l)!l!\left(\frac{N-3}{2}+l\right)_l
\left(\sum_{n=\left \lceil{m/2}\right
    \rceil}^{m-l}2^{2n-m+l}\frac{(-1)^n}{2n}{\binom{n}{m-n}}{\binom{m-n}{l}}\right)^2\,,\\
\intertext{and secondly}
    \label{eq:7}
    \lambda_N^{s,m}&=m!\sum_{l=0}^{\left \lfloor{m/2}\right \rfloor}
    (m-2l)!l!\left(\frac{N-3}{2}+l\right)_l
\left(\sum_{n=\left \lceil{m/2}\right
    \rceil}^{m-l}2^{2n-m+l}\binom{\frac{s}{2}}{n}{\binom{n}{m-n}}{\binom{m-n}{l}}\right)^2\,.
  \end{align}

In the next proposition we construct a family of smooth maps that will be useful for the proof of Theorems \ref{thm2} and \ref{thm1}. The same construction was used in \cite{Shafrir:2017}.
\begin{proposition}
  \label{prop:prop}
There exists a family of maps
$\{u_\varepsilon\}_{\varepsilon>0}\subset C^\infty_c(\R^N)$ satisfying
\begin{flalign}
u_\varepsilon(x)=\log(1/|x|) \text{ on }B_1\setminus B_\varepsilon,\label{eq:43}\\
\|u_\varepsilon\|_{L^\infty(\R^N)}=u_\varepsilon(0)=\log(1/\varepsilon)+O(1),\label{eq:44}\\
\text{supp}(u_\varepsilon)\subset B_2,\label{eq:45}\\
\big\||\nabla^k u_\varepsilon|\big\|_{L^\infty(B_\varepsilon)}=
  O(\varepsilon^{-k}),\,1\leq k\leq N,\label{eq:46}\\
\big\||\nabla^k u_\varepsilon|\big\|_{L^\infty(B_2\setminus B_1)}=O(1), \,1\leq k\leq N.\label{eq:47}
\end{flalign}
\end{proposition}
\begin{proof}
Let $\varphi\in C^\infty[0,\infty)$ satisfy $\varphi\equiv
0$ on $[0,1/2]$, $\varphi\equiv 1$ on
$[1,\infty)$ and $\varphi(t)\in[0,1]$ for all $t$. For each
$\varepsilon>0$ define $\varphi_\varepsilon(t)=\varphi(t/\varepsilon)$ on
$[0,\infty)$. Clearly,
\begin{equation}
  \label{eq:17}
  \|\varphi_\varepsilon^{(j)}\|_{L^\infty(0,\infty)}\leq\frac{C_j}{\varepsilon^j},\quad
  \forall j\ge 1.
\end{equation}
Set on $[0,\infty)$,
$f_\varepsilon(t)=(-\log\varepsilon)-\int_{\varepsilon}^t \frac{\varphi_\varepsilon (s)}{s}\,ds$.
Finally, let
 $\zeta\in C^\infty_c(\R^N)$ be a cut-off function satisfying
 $supp(\zeta)\subset B_2(0)$ and
 $\zeta=1$ on $B_1(0)$ and set
 $u_\varepsilon(x)=\zeta(x)f_\varepsilon(|x|)$ on $\R^N$. The validity
 of \eqref{eq:43},\eqref{eq:45} and \eqref{eq:47} is clear from the
 definition. 
Combining \eqref{eq:17} with the estimates (see \eqref{log_equality})
\begin{equation}
  \label{eq:20}
  \big\| |\nabla^k\log |x|\big\|_{L^\infty(B_{\varepsilon}\setminus
    B_{\varepsilon/2})}\leq \frac{C_k}{\varepsilon^k}, \,1\leq k\leq N,
\end{equation}
 with \eqref{eq:17} yields \eqref{eq:46}. Finally, \eqref{eq:44}
 follows from the case $k=1$ in \eqref{eq:46}, together with \eqref{eq:43}. 

\end{proof}

\section{PDE and Potential Estimates}\label{PDEs}

In this Section we prove two versions of \eqref{log_PDE} with precise
constant, as well as some useful corollaries.  Let us first give the
following result, which extends \cite[Proposition~3.1]{Shafrir:2017} from $k=N$ to any $k \in \mathbb{N}$.

\begin{theorem}\label{thm1}
For $k \in \mathbb{N}$ one has the equality
\begin{align*}
(-1)^{k}\div_k  \left(|x|^{2k-N} \nabla^k \log |x|\right) = -\ell^k_N \SNO \delta_0
\end{align*}
in the sense of distributions.
\end{theorem}
A useful consequence of Theorem \ref{thm1} that we will require in the
sequel is the following corollary which establishes a potential estimate for a function in terms of its higher order gradient.
\begin{corollary}\label{endpoint_derivative}
Let $m \in \mathbb{N}$, $m<N$.  Then one has
\begin{equation*}
|u(x)| \leq \frac{\gamma(m)}{\sqrt{\ell^m_N}\SNO} I_m(|\nabla^m u|)(x)\,,~\forall u\in C^\infty_c(\R^N).
\end{equation*}
\end{corollary}
Let us quickly prove the corollary before we return to give the proof of Theorem \ref{thm1}.
\begin{proof}
As a result of Theorem \ref{thm1}, if $u \in C^\infty_c(\mathbb{R}^N)$ we have
\begin{align*}
u(x) = \frac{1}{\ell^m_N\SNO} \int_{\mathbb{R}^N} |y|^{2m-N} \nabla^m u(x-y)\cdot \nabla^m \log (1/|y|) \;dy,
\end{align*}
and using finite dimensional Cauchy-Schwarz and the relation \eqref{log_equality} we find
\begin{align*}
|u(x)| &\leq \frac{1}{\ell^m_N\SNO} \int_{\mathbb{R}^N} |y|^{2m-N} |\nabla^m u(x-y)| \frac{\sqrt{\ell^m_N}}{|y|^m}\;dy \\
&= \frac{\gamma(m)}{\sqrt{\ell^m_N}\SNO} I_m |\nabla^m u|,
\end{align*}
which is the desired result.
\end{proof}

\begin{proof}[Proof of Theorem \ref{thm1}]
We argue as in the proof in \cite[Proposition~3.1]{Shafrir:2017}.  Thus, we define the function
  \begin{equation}
\label{eq:10}
    F(x):=(-1)^k \div_k \left(|x|^{2k-N}  \nabla^k \log|x|\right),
  \end{equation}
  which again belongs to $C^\infty(\R^N\setminus\{0\})$. 
As in \cite{Shafrir:2017}, it is easy to verify that $F$ is a {\em
  radial} function which is homogenous of
 degree $-N$. Therefore it must be of the form
 \begin{equation}
   \label{eq:11}
   F(x)=c|x|^{-N},
 \end{equation}
 for some constant $c\in\R$. We claim that $c=0$. 
\par Assume by contradiction that $c\neq 0$. For each $\varepsilon\in(0,1)$  set
 $v_\varepsilon=\zeta\varphi_\varepsilon\in C^\infty_c(\R^N)$, with the same $\zeta$ and
 $\varphi_\varepsilon(t)=\varphi(t/\varepsilon)$  as defined in the course of
 the proof of Proposition~\ref{prop:prop}. It is easy to verify that 
 \begin{equation}
   \label{eq:13}
   \int_{\R^N}|\nabla^kv_\varepsilon|\leq C,~\text{ uniformly in }\varepsilon.
 \end{equation}
On the other hand, by \eqref{log_equality} we have 
\begin{equation}
\label{eq:26}
|x|^{2k-N}
\big|\frac{1}{\sqrt{\ell^k_N}}\nabla^k\log|x|\big|=1,\text{ for all }x\ne0,
\end{equation}
whence
\begin{equation}
  \label{eq:9}
  \int_{\R^N}|\nabla^kv_\varepsilon|\geq
  \int_{B_2\setminus B_{\varepsilon/2}}|\nabla^kv_\varepsilon|\geq 
\frac{1}{\sqrt{\ell^k_N}}\left|\int_{B_2\setminus B_{\varepsilon/2}}\!\!\!\left(\nabla^kv_\varepsilon\right)\cdot\left(|x|^{2k-N}\nabla^k\log|x|\right)\right|\,.
\end{equation}
Applying integration by parts to the integral on  the R.H.S.~of \eqref{eq:9} and using
\eqref{eq:11} gives
\begin{multline}
  \label{eq:14}
\frac{1}{\sqrt{\ell^k_N}}\int_{B_2\setminus B_{\varepsilon/2}}\left(\nabla^kv_\varepsilon\right)\cdot\left(|x|^{2k-N}\nabla^k\log|x|\right)=\frac{1}{\sqrt{\ell^k_N}}\int_{B_2\setminus B_{\varepsilon/2}}Fv_\varepsilon\\=
\frac{c}{\sqrt{\ell^k_N}}\int_{B_2\setminus B_{\varepsilon/2}}\frac{v_\varepsilon}{|x|^N}=\frac{c}{\sqrt{\ell^k_N}}\int_{B_1\setminus B_{\varepsilon}}\frac{dx}{|x|^N}+O(1)=
\frac{c}{\sqrt{\ell^k_N}}\,\SNO\log(1/\varepsilon)+O(1)\,,
\end{multline}
where $O(1)$ denotes a bounded quantity,
 uniformly in $\varepsilon$.
Combining \eqref{eq:9}--\eqref{eq:14} with \eqref{eq:13} leads to a
contradiction for $\varepsilon$ small enough, whence $c=0$ as claimed.
\par From the above we deduce that the {\em distribution} 
\begin{equation}
\label{eq:15}
  \mathcal{F}:=(-1)^k \div_k \left(|x|^{2k-N}  \nabla^k \log|x|\right) \in \mathcal{D}'(\R^N)
\end{equation}
satisfies $supp(\mathcal{F})\subset\{0\}$. By a celebrated theorem of
L.~Schwartz~\cite{sch} it follows that
\begin{equation}
  \label{eq:16}
  \mathcal{F}=\sum_{j=1}^Lc_jD^{\alpha_j} \delta_0\,,
\end{equation}
for some multi-indices  $\alpha_1,\ldots,\alpha_L$.
But by \eqref{eq:5} the R.H.S.~of \eqref{eq:15} can be
written as
\begin{equation*}
(-1)^k\div_k \left(|x|^{2k-N}  \nabla^k \log|x|\right)= \div G,
\end{equation*}
with $G=(G_1,\ldots,G_N):\R^N\to\R^N$ satisfying $|G_j(x)|\leq
C/|x|^{N-1}$, for all $j$. Hence 
$G\in L^1_{\text{loc}}(\R^N,\R^N)$. It follows that $\mathcal{F}$ in
\eqref{eq:15} is a sum of first derivatives of functions in
$L^1_{\text{loc}}$, whence for some $\mu\in\R$,
\begin{equation}
  \label{eq:41}
  \mathcal{F}=\mu\delta_0\,.
\end{equation}
\par It remains to determine the value of $\mu$ in \eqref{eq:41}. For
that matter we use the test functions
$\{u_\varepsilon\}$ given by Proposition~\ref{prop:prop}.
By \eqref{eq:15} and \eqref{eq:41}  we have
 \begin{equation}
   \label{eq:21}
   \mu u_\varepsilon(0)=\int_{\R^N}|x|^{2k-N}\left(\nabla^k
     u_\varepsilon\right)\cdot\left(\nabla^k\log |x|\right)\,.
 \end{equation}
By \eqref{eq:43}--\eqref{eq:47} we get for the R.H.S.~of \eqref{eq:21},
\begin{multline}
  \label{eq:23}
  \int_{\R^N}|x|^{2k-N}\left(\nabla^k
     u_\varepsilon\right)\cdot\left(\nabla^k\log |x|\right)=-\int_{\{\varepsilon<|x|<1\}}|x|^{2k-N}\left(\nabla^k
     \log |x|\right)\cdot\left(\nabla^k\log |x|\right)+O(1)\\=-\ell^k_N \int_{\{\varepsilon<|x|<1\}}\frac{dx}{|x|^N}+O(1)=-\ell^k_N\SNO(-\log\varepsilon)+O(1)\,.
\end{multline}
 On the other hand, for the L.H.S.~of
 \eqref{eq:21} we have by \eqref{eq:43} that 
 \begin{equation}
   \label{eq:24}
    \mu u_\varepsilon(0)=-\mu\log\varepsilon+O(1).
 \end{equation}
Plugging \eqref{eq:23}--\eqref{eq:24} in \eqref{eq:21} yields
$\mu=-\ell^k_N\SNO$ as claimed.
\end{proof}

For the more general family of inequalities we prove in this paper, we require the following version of \eqref{log_PDE} in the regime $\alpha>k$. 
\begin{theorem}\label{thm0}
For every $k\in \mathbb{N}$, $\alpha \in (k,\infty)$, we have 
\begin{equation}
\label{eq:37}
(-1)^k \div_k  \left(|x|^{2\alpha-N} \nabla^k \left( \frac{1}{|x|^{\alpha-k}} \right)\right) = \frac{\lambda_N^{k-\alpha,k}}{|x|^{N-\alpha+k}}\,.
\end{equation}
\end{theorem}

To see that Theorem \ref{thm0} is a version of \eqref{log_PDE}, we note first that by \rlemma{lem:lap-log} we have 
\begin{align*}
(-\Delta)^{(\alpha-k)/2} \log |x| = -\left(\frac{\gamma(\alpha-k)}{\SNO}\right)\frac{1}{|x|^{\alpha-k}},
\end{align*}
and  invert the outermost fractional Laplacian $(-\Delta)^{(\alpha-k)/2}$ on the L.H.S.~of \eqref{log_PDE}.

Again, before we prove Theorem \ref{thm0}, let us record and prove another useful sharp potential representation, the following 
\begin{corollary}\label{intermediate_derivatives}
Let $m \in \mathbb{N}$, $m<N$ and $k \in \{1,\ldots, m-1\}$.  Then one has
\begin{equation*}
  |u(x)|\leq \frac{\gamma(m)}{\gamma(m-k)\sqrt{\lambda_N^{k-m,k}}} I_m |\nabla^k (-\Delta)^{(m-k)/2}u|.
\end{equation*}
\end{corollary}
\begin{proof}
By the semi-group property of the Riesz potentials and fractional Laplacian for $u \in C^\infty_c(\mathbb{R}^N)$ we have
\begin{align*}
u(x) = I_{m-k} (-\Delta)^{(m-k)/2}u = \frac{1}{\gamma(m-k)}\left(\frac{1}{|\cdot|^{N-(m-k)}} \ast (-\Delta)^{(m-k)/2}u\right)(x).
\end{align*}
An application of Theorem \ref{thm0} with $\alpha=m$ then yields
\begin{multline*}
\left(\frac{1}{|\cdot|^{N-(m-k)}} \ast (-\Delta)^{(m-k)/2}u\right)(x) \\
= \frac{1}{\lambda_N^{k-m,k}} \int_{\mathbb{R}^N} |y|^{2m-N} \nabla^k (-\Delta)^{(m-k)/2}u(x-y)\cdot \nabla^k \frac{1}{|y|^{m-k}} \;dy,
\end{multline*}
and thus
\begin{align*}
u(x) = \frac{1}{\gamma(m-k)\lambda_N^{k-m,k}}\int_{\mathbb{R}^N} |y|^{2m-N} \nabla^k (-\Delta)^{(m-k)/2}u(x-y)\cdot \nabla^k \frac{1}{|y|^{m-k}} \;dy.
\end{align*}
Then similarly by the finite dimensional Cauchy-Schwarz and the relation \eqref{eq:5} we find
\begin{align*}
|u(x)| &\leq \frac{1}{\gamma(m-k)\lambda_N^{k-m,k}}\int_{\mathbb{R}^N} |y|^{2m-N} |\nabla^k (-\Delta)^{(m-k)/2}u(x-y)| \frac{\sqrt{\lambda^{k-m,k}_N}}{|y|^m}\;dy \\
&= \frac{\gamma(m)}{\gamma(m-k)\sqrt{\lambda_N^{k-m,k}}} I_m |\nabla^k (-\Delta)^{(m-k)/2}u|(x),
\end{align*}
which is the desired result.
\end{proof}
\begin{proof}[Proof of Theorem~\ref{thm0}]
 Let the L.H.S.~of \eqref{eq:37} be denoted by $F(x)$, i.e.,
 \begin{equation}
 \label{eq:38}
    F(x) =(-1)^k \div_k  \left(|x|^{2\alpha-N} \nabla^k \left( \frac{1}{|x|^{\alpha-k}} \right)\right) .
  \end{equation}
By the same argument as in the proof of Theorem~\ref{thm1}, $F$ is a radial and homogenous of order $\alpha-N-k$. From \eqref{eq:38} we  deduce easily that $|F(x)|\leq C|x|^{\alpha-N-k}\in L^1_{\text{loc}}(\R^N)$, whence 
$F(x)=\mu|x|^{\alpha-N-k}$ for some constant $\mu$. To finish the proof of \eqref{eq:37} it remains
to determine the value of $\mu$.
\par For each $\varepsilon\in(0,1)$ let $v_\varepsilon\in
C^\infty_c(\R^N)$ be the function as defined in the proof of \rth{thm1}. 
We multiply the equation
\begin{equation*}
(-1)^k \div_k  \left(|x|^{2\alpha-N} \nabla^k \left( \frac{1}{|x|^{\alpha-k}} \right)\right) = \mu |x|^{\alpha-N-k}
\end{equation*}
by $v_\varepsilon |x|^{k-\alpha}$ and integrate over $\R^N$. For the
R.H.S.~we get
\begin{equation}
  \label{eq:25}
  \mu\int_{\R^N} |x|^{\alpha-N-k}
  \cdot v_\varepsilon |x|^{k-\alpha}=\mu\int_{B_1\setminus B_\varepsilon}|x|^{-N}+O(1)=\omega_{N-1}\mu\log(1/\varepsilon)+O(1).
\end{equation}
For the L.H.S.~we obtain, using the properties of $v_\varepsilon$ and \eqref{eq:5},
\begin{multline}
  \label{eq:27}
  \int_{\R^N}  (-1)^k \div_k  \left(|x|^{2\alpha-N} \nabla^k
    |x|^{k-\alpha}\right)\left(v_\varepsilon|x|^{k-\alpha}\right)\\= \int_{\R^N}  \left(|x|^{2\alpha-N} \nabla^k
    |x|^{k-\alpha}\right)\cdot\nabla^k\left(v_\varepsilon|x|^{k-\alpha}\right)
= \int_{B_1\setminus B_{\varepsilon}}\!\!|x|^{2\alpha-N}\Big|\nabla^k(|x|^{k-\alpha})\Big|^2+O(1)\\=\omega_{N-1}\lambda_N^{k-\alpha,k}\log(1/\varepsilon)+O(1).
\end{multline}
Equating \eqref{eq:25} to \eqref{eq:27} yields
$\mu=\lambda_N^{k-\alpha,k}$, as claimed.
  \end{proof}

\section{Proofs of the main results} \label{proofs}
We begin with the result that provides embedding of the spaces $X^{N,k}$ ($k=1,\ldots,N-1$) in $L^\infty(\R^N)$, with best constant when $N-k$ is even.
  \begin{proof}[Proof of Theorem~\ref{thm2}]
  We start with \eqref{eq:37}, for $\alpha=N$, to which we apply convolution with the
    function $(-\Delta)^{(N-k)/2}u\in\mathcal{S}$. For the L.H.S.~we
    get, applying the Cauchy-Schwarz inequality and \eqref{eq:5},
    \begin{multline}
      \label{eq:28}
     \int_{\R^N} |y|^{N} \nabla^k |y|^{k-N}\cdot\nabla^k
    \left ( (-\Delta)^{(N-k)/2}u (x-y)\right)\,dy\leq \\\left(\lambda_N^{k-N,k}\right)^{1/2}\int_{\R^N}\left|\nabla^k
      (-\Delta)^{(N-k)/2}u\right|.
    \end{multline}
For the R.H.S., we deduce using the Fourier transform, see
\cite{Stein:1970}*{p. 117}, that
\begin{equation}
  \label{eq:29}
  \lambda_N^{k-N,k}|x|^{-k}*\big((-\Delta)^{(N-k)/2}u\big)= \lambda_N^{k-N,k}\gamma(N-k)u(x).
\end{equation}
Combining \eqref{eq:28} and \eqref{eq:29} finally yields
\begin{equation*}
  |u(x)|\leq {\left(\lambda_N^{k-N,k}\right)}^{-1/2}{\gamma(N-k)}^{-1}\int_{\R^N}\left|\nabla^k
      (-\Delta)^{(N-k)/2}u\right|.
\end{equation*}
\par  Next we prove optimality of the constant $c_k$ in
\eqref{k_inequality_star} in the case where $N-k$ is even.
 Consider $u_\varepsilon$ as given by  Proposition~\ref{prop:prop}. 
 By \eqref{eq:5}, \eqref{eq:1} and \eqref{eq:43} we get,
 \begin{equation}
   \label{eq:31}
  \left|\nabla^k\left( (-\Delta)^{(N-k)/2}u_\varepsilon\right)\right|(x)=\left(\frac{\gamma(N-k)}{\omega_{N-1}}\right)\left(\lambda_N^{k-N,k}\right)^{1/2}|x|^{-N}\text{ for }\varepsilon<|x|<1.
 \end{equation}
 From \eqref{eq:31} we deduce, taking into account
 \eqref{eq:45}--\eqref{eq:47}, that
 \begin{multline}
   \label{eq:32}
   \int_{\R^N}\left|\nabla^k\left(
       (-\Delta)^{(N-k)/2}u_\varepsilon\right)\right|=\int_{B_1\setminus
   B_\varepsilon}\left|\nabla^k\left(
     (-\Delta)^{(N-k)/2}u_\varepsilon\right)\right|+O(1)\\
=\gamma(N-k)|\left(\lambda_N^{k-N,k}\right)^{1/2}\log(1/\varepsilon)+O(1).
 \end{multline}
Combining \eqref{eq:32} with \eqref{eq:43} yields
\begin{equation}
  \label{eq:33}
  \lim_{\varepsilon\to0} \frac{\sup_{x\in\R^N}
|u_\varepsilon(x)|}{\int_{\R^N}\left|\nabla^k\left(
       (-\Delta)^{(N-k)/2}u_\varepsilon\right)\right|}={\left(\lambda_N^{k-N,k}\right)^{-1/2}}\gamma(N-k)^{-1}\,.
\end{equation}
\end{proof}  
We continue with the results related to Adams inequality.   
\begin{proof}[Proof of \rth{thm4}]
Clearly it suffices to consider $u\in C^\infty_c(\Omega)$. We begin with the assertion given in Corollary \ref{endpoint_derivative}, which is the estimate
\begin{align*}
|u(x)| \leq \frac{\gamma(m)}{\sqrt{\ell^m_N}\SNO} I_m |\nabla^m u|.
\end{align*}
Then an application of Theorem \ref{Adams} with $f= |\nabla^m u|$ and $\alpha=m$ yields the estimate
\begin{align*}
\int_\Omega \exp\left (\frac{N}{\omega_{N-1}}\gamma(m)^{p^\prime} \left| \frac{I_m |\nabla^m u|}{\||\nabla^m u|\|_{L^p}}\right|^{p^\prime}\right)\;dx \leq A|\Omega|,
\end{align*}
which by monotonicity of the exponential implies
\begin{align*}
\int_\Omega \exp\left (\frac{N}{\omega_{N-1}} (\sqrt{\ell^m_N}\SNO)^{p^\prime}\left| \frac{u(x)}{\||\nabla^m u|\|_{L^p}}\right|^{p^\prime}\right)\;dx \leq A |\Omega|.
\end{align*}
This demonstrates the estimate holds with 
\begin{align*}
\widetilde\beta_0(m,N) = N (\ell^m_N)^{p^\prime/2} \SNO^{p^\prime-1}=N{\ell^m_N}^{\frac{N}{2(N-m)}} \SNO^{\frac{m}{N-m}}\,.
\end{align*}

We now show optimality of this constant $\widetilde\beta_0(m,N)$.  Without loss of generality we may take $\Omega = B_2$ (otherwise the construction can be translated and dilated appropriately).  Assume  that
 for some $\beta$ the inequality
\begin{equation}
\label{eq:B4}
\int_{B_2} \exp\left (\beta \left| \frac{u(x)}{\||\nabla^m u|\|_{L^p}}\right|^{p^\prime}\right)\;dx \leq A|B_2|\,,
\end{equation}
holds for all $u\ne0$ in $W^{m,p}_0(B_2)$. Our objective is to show that necessarily $\beta\le\widetilde\beta_0(m,N)$. Taking the test functions $u_\varepsilon$ constructed in Proposition \ref{prop:prop}, we observe
\begin{equation}
\label{eq:ump}
\begin{aligned}
\| |\nabla^m u_\varepsilon| \|^p_{L^p} &=  \int_{B_1\setminus B_\varepsilon} |\nabla^m u_\varepsilon|^p + O(1)\\
& = \int_{B_1 \setminus B_\varepsilon} |\nabla^m \log |x| |^p + O(1) 
= \left(\ell^m_N\right)^{p/2} \int_{B_1 \setminus B_\varepsilon} \frac{1}{|x|^N}+O(1)\\
&\phantom{a}\leq  (1+\delta)^{p/p^\prime}\left(\ell^m_N\right)^{p/2}\SNO\log \frac{1}{2\varepsilon}\,,
\end{aligned}
\end{equation}
for every $\delta>0$ and $\varepsilon\leq\varepsilon_0(\delta)$.
 Applying \eqref{eq:B4} with $u=u_\varepsilon$, taking into account \eqref{eq:ump} and the fact that on $B_{2\varepsilon}\setminus B_\varepsilon$ we have $u_\varepsilon(x)=\log 1/|x|\geq \log 1/(2\varepsilon)$, yields  
\begin{equation}
\label{eq:dmp}
\begin{aligned}
A|B_2|&\geq \int_{B_{2\varepsilon}\setminus B_\varepsilon}\exp\left (\beta \left| \frac{u_\varepsilon(x)}{\||\nabla^m u_\varepsilon|\|_{L^p}}\right|^{p^\prime}\right)\;dx\\
&\geq C(N)\varepsilon^N\exp\left(\frac{\beta(\log\frac{1}{2\varepsilon})^{p^\prime}}{(1+\delta)\left(\ell^m_N\right)^{p^\prime/2}\SNO^{p^\prime/p}(\log\frac{1}{2\varepsilon})^{p^\prime/p}}   \right)
\\ &=\frac{C(N)}{2^N}\left(\frac{1}{2\varepsilon}\right)^{\frac{\beta}{(1+\delta)\left(\ell^m_N\right)^{p^\prime/2}\SNO^{p^\prime/p}}-N}\;.
\end{aligned}
\end{equation}
Letting $\varepsilon\to0$ we deduce from \eqref{eq:dmp} that we must have
\begin{equation*}
\beta\leq N(1+\delta)\left(\ell^m_N\right)^{p^\prime/2}\SNO^{p^\prime/p}.
\end{equation*}
Since $\delta>0$ is arbitrary, we finally obtain that $\beta\leq\widetilde\beta_0(m,N)$.
\end{proof}

We now proceed to prove Theorem \ref{thm5}.  The proof follows in analogy with that of Theorem \ref{thm4}, though we provide the details with precise track of the constants for the convenience of the reader.  
\begin{proof}
Here we begin with the inequality established in Corollary \ref{intermediate_derivatives}, which asserts
\begin{align*}
|u(x)| \leq \frac{\gamma(m)}{\gamma(m-k)\sqrt{\lambda_N^{k-m,k}}} I_m |\nabla^k (-\Delta)^{(m-k)/2}u|.
\end{align*}
Then in this setting we invoke Theorem \ref{Adams} with $f= |\nabla^k (-\Delta)^{(m-k)/2}u|$ and $\alpha=m$ to obtain the estimate
\begin{align*}
\int_\Omega \exp\left (\frac{N}{\omega_{N-1}}\gamma(m)^{p^\prime} \left| \frac{I_m |\nabla^k (-\Delta)^{(m-k)/2}u|}{\||\nabla^k (-\Delta)^{(m-k)/2}u|\|_{L^p}}\right|^{p^\prime}\right)\;dx \leq A|\Omega|\,,
\end{align*}
while again monotonicity of the exponential then implies
\begin{align*}
\int_\Omega \exp\left (\frac{N}{\omega_{N-1}} \Big(\gamma(m-k)\sqrt{\lambda_N^{k-m,k}}\Big)^{p^\prime}\left| \frac{u(x)}{\||\nabla^k (-\Delta)^{(m-k)/2}u|\|_{L^p}}\right|^{p^\prime}\right)\;dx \leq A|\Omega|\,.
\end{align*}
This establishes the estimate with 
\begin{align*}
\widetilde\beta_0(m,k,N) = \frac{N}{\omega_{N-1}} \Big(\gamma(m-k)\sqrt{\lambda_N^{k-m,k}}\Big)^{p^\prime}.
\end{align*}

The proof of the optimality of the constant $\widetilde\beta_0(m,k,N)$ follows by the same argument used in the proof of \rth{thm4}, i.e., using the functions $\{u_\varepsilon\}$ on $B_2$; the details are left to the interested reader. 
\end{proof}

Finally, we conclude with the proof of Proposition \ref{thm3}.

\begin{proof}[Proof of Proposition~\ref{thm3}]
By approximation in the strict topology on $M_b(\mathbb{R}^N)$, it suffices to prove the result for $u \in C^\infty_c(\mathbb{R}^N)$.  Thus, we begin by using the duality of $\mathcal{H}^1(\mathbb{R}^N)$ and $BMO(\mathbb{R}^N)$ to realize the semi-norm of $u \in C^\infty_c(\mathbb{R}^N)$ as
\begin{align}
|u|_{BMO} &= \sup_{f \in \mathcal{H}^1(\mathbb{R}^N), \|f\| \leq 1} \int_{\mathbb{R}^N} uf. \label{BMO-duality}
\end{align}
Now by the semi-group property of the Riesz potentials and fractional Laplacian we find
\begin{align*}
\int_{\mathbb{R}^N} uf\,dx = \int_{\mathbb{R}^N} \big(I_{N-1} (-\Delta)^{N/2}u\big) \big(I_1f\big)\,dx.
\end{align*}
Let us expand the term $I_{N-1} (-\Delta)^{N/2}u$, where we find
\begin{align*}
I_{N-1} (-\Delta)^{N/2}u &= \frac{1}{\gamma(N-1)} \int_{\mathbb{R}^N} (-\Delta)^{N/2} u(y) \frac{1}{|x-y|}\;dy \\
&=\frac{1}{(N-1)\gamma(N-1)} \int_{\mathbb{R}^N} (-\Delta)^{N/2} u(y) \div_x \left(\frac{x-y}{|x-y|}\right)\;dy.
\end{align*}
However, now performing an integration by parts and Fubini's theorem we have
\begin{align*}
 \int_{\mathbb{R}^N} uf\,dx= -\frac{1}{(N-1)\gamma(N-1)} \int_{{\R}^N} \Big\{Rf(x) \cdot  \int_{\mathbb{R}^N} (-\Delta)^{N/2} u(y) \frac{x-y}{|x-y|}\,dy\Big\}\,dx
\end{align*}
Thus by the finite dimensional Cauchy Schwartz inequality we deduce
\begin{align}
\left|\int_{\mathbb{R}^N} uf\;dx\right| \leq  c_1\|(-\Delta)^{N/2}u\|_{L^1(\mathbb{R}^N)} \|Rf\|_{L^1(\mathbb{R}^N)}. \label{estimate-2}
\end{align}
As
\begin{align*}
 \|Rf\|_{L^1(\mathbb{R}^N)} \leq  \int_{\mathbb{R}^N} |(f,Rf)| = \|f\|_{\mathcal{H}^1(\mathbb{R}^N)},
\end{align*}
we find that \eqref{BMO-duality} and \eqref{estimate-2} yield the desired result with
\begin{align*}
c_0 = \frac{1}{(N-1)\gamma(N-1)}\; (=c_1=\frac{1}{\widetilde{\gamma}(N-1)}).
\end{align*}  
\end{proof}

 \section*{Acknowledgements}
I.S. is supported by the Israel Science Foundation (Grant No. 999/13).  D.S. is supported by the Taiwan Ministry of Science and Technology under research grant 105-2115-M-009-004-MY2.  Part of this work was written while I.S. was visiting National Chiao Tung University with support from the National Center for Theoretical Sciences.  He would like to thank NCTU for its warm hospitality and the NCTS for its support to visit.

\end{document}